\documentclass[oneside, 10pt]{amsart}
\usepackage{preamble}

\numberwithin{equation}{section}

\begin{document}

\title[]{Direct limits of large orbits and the Knaster continuum homeomorphism group}

\author{Sumun Iyer}
\address{Cornell University}
\email{ssi22@cornell.edu}
\blfootnote{Research supported by NSF GRFP grant DGE – 2139899}

\begin{abstract}
The main result is that the group $\Homeo (K)$ of homeomorphisms of the universal Knaster continuum contains an open subgroup with a comeager conjugacy class. Actually, this open subgroup is the very natural subgroup consisting of \emph{degree-one} homeomorphisms.  We give a general fact about finding comeager orbits in Polish group actions which are approximated densely by direct limits of actions with comeager orbits. The main theorem comes as a result of this fact and some finer analysis of the conjugacy action of the group $\Ho$.
\end{abstract}

\maketitle

\bigskip


\noindent 

\section{Introduction}
This paper concerns a well-studied property of some Polish groups--the existence of a comeager conjugacy class, sometimes also called the \emph{strong topological Rokhlin property} (see \cite{glasnerweiss} for a survey on the topic). This is a very strong property to have, it says that from a topological perspective basically the entire group is a single conjugacy class. It is exclusively a large Polish group phenomenon--non-trivial locally compact Polish groups cannot have a comeager conjugacy class (see \cite{wesolek}). There is a well-developed theory of this phenomenon for non-archimedean Polish groups (these are closed subgroups of the infinite permutation group $S_\infty$). Work of Truss and then of Kechris and Rosendal give combinatorial properties which characterize exactly when a non-archimedean Polish group has a dense or comeager conjugacy class (see \cite{truss}, \cite{kr}). There are by now many examples of non-archimedean groups known to have or not have a comeager conjugacy class, see \cite{truss}, \cite{kr}, \cite{bcmg1}, \cite{bcmg2}, \cite{kusketruss}. 

Our interest here is in homeomorphism groups of connected spaces and in this setting much less is known. In fact the only known result concerning existence or non-existence of a comeager conjugacy class for groups of this sort appears to be the following:

\begin{thm}[Hjorth \cite{hjorth}]
    The group of homeomorphisms of the interval contains an open subgroup with a comeager conjugacy class. Actually, the open subgroup is the group $\Ho$ of order-preserving homeomorphisms.
\end{thm}

The goal of the present paper is to provide an example of a much more complicated connected space whose homeomorphism group contains an open subgroup with comeager conjugacy class. Knaster continua are a class of indecomposable continua. A continuum is \emph{indecomposable} if it cannot be written as the union of two proper, non-trivial subcontinua. Indecomposability is, generally, an indicator of complexity of the continuum (most basic continua like the arc $[0,1]$ and cubes $[0,1]^n$ are decomposable). We focus on the \emph{universal Knaster continuum}, $K$-- where \emph{universal} means that $K$ continuously and openly surjects onto every other Knaster continuum (precise definitions may be found in Section \ref{secbasic}). Equipped with the uniform-convergence topology, $\Homeo(K)$ is a non-locally compact Polish group. Here is the main theorem of the paper:
\begin{thm}\label{thm_0}
    The group $\Homeo(K)$ contains an open subgroup with a comeager conjugacy class. 
\end{thm}
Theorem \ref{thm_0} is proven in Section \ref{sec:main_thm} (as Theorem \ref{thm_main}). Just as for $\Homeo[0,1]$, the open subgroup in Theorem \ref{thm_0} is a natural one. It is the subgroup of \emph{degree-one homeomorphisms} of $K$, which we denote by $\Hk$. There is a notion of degree for homeomorphisms of Knaster continua due to D\k{e}bski which is conjugacy invariant \cite{debski}. For now, it ought to be thought of loosely as something like rotation number for circle homeomorphisms, precise definitions will be given in Section \ref{secbasic}.

Now we say more about the proof of Theorem \ref{thm_0}. The proof depends crucially on the fact that there is a particular countable chain of copies of $\Homeo_+[0,1]$ inside of $\Homeo^1(K)$. These copies sit in $\Homeo^1(K)$ as small (nowhere-dense) sets and the union of the copies of $\Ho$--call this union $H$ for the moment--is dense and meager in $\Homeo^1(K)$. In spite of the fact that $H$ is topologically small inside $\Homeo^1(K)$, the remnant of genericity from $\Homeo_+[0,1]$ propagated ``up the countable chain'' is enough to get Theorem \ref{thm_0}. 

Actually Theorem \ref{thm_0} can be phrased as a consequence of a general fact about Polish group actions which are densely approximated by direct limits of actions with comeager orbits that satisfy certain metric conditions. Checking these metric conditions is non-trivial. For our application, it involves (1) quantitive analysis of how elements of $\Ho$ behave under conjugacy by elements close to the identity and (2) a lemma about orbits of points under the tent map dynamical system (Lemma \ref{lem3}).

The paper is structured as follows: Section \ref{sec:gen_crit} formulates the general criterion about direct limits of Polish group actions with comeager orbits. Section \ref{secbasic} has some basic facts about Knaster continua. Section \ref{sec:homeo_int} analyzes the conjugacy action in $\Ho$ and Section \ref{sec:main_thm} contains the proof of Theorem \ref{thm_0}.

\subsection{Acknowledgements:} I would like to thank Sławek Solecki for many helpful conversations and suggestions related to this project.

\section{Comeager orbits in direct limits of Polish group actions}\label{sec:gen_crit}

This section gives a general strategy for producing a comeager orbit in a Polish group action which is approximated by a direct limit. Given a Polish group $G$, a \emph{Polish $G$-space} is a Polish space $X$ equipped with a continuous action of $G$. The $G$-space is \emph{topologically transitive} if for any $U,V$ open nonempty subsets of $X$, there is $g \in G$ such that $g \cdot U \cap V \neq \emptyset$. Of course if $X$ has a dense orbit, then it is topologically transitive.

First, a condition for comeager orbits due to Rosendal. A proof may be found in \cite{bymt}, Proposition 3.2. 

\begin{lem} [Rosendal] \label{lem_rosendal}
    Let $X$ be a topologically transitive Polish $G$-space. The following are equivalent:
    \begin{enumerate}
        \item $X$ has a comeager orbit
        \item For any open $V \ni 1_G$ and open, nonempty $U \subset X$, there exists non-empty $U' \subseteq U$ so that for any $W_1,W_2$ nonempty open subsets of $U'$ we have $V \cdot W_1 \cap W_2 \neq \emptyset$.
    \end{enumerate}
\end{lem}

The condition below is most convenient for what we want to prove in this section; it it nearly just a restatement of Lemma \ref{lem_rosendal} above.

\begin{lem}\label{lem_dense_condition}
    A Polish $G$-space $X$ has a comeager orbit if and only if there is a dense set $D \subseteq X$ such that $D$ is contained in a single orbit and for any open $V \ni 1_G$ and any $x \in D$, $V \cdot x$ is dense in a neighborhood of $x$.
\end{lem}

\begin{proof}
    The forward direction is immediate from Effros' theorem (\cite{effros}, Theorem 2.1). In particular, for every $y$ in the co-meager orbit and every open neighborhood of the identity $W$, $W \cdot y$ is co-meager (and so dense) in a neighborhood of $y$. 

    For the reverse direction, suppose that $X$ satisfies the condition given in the statement of the lemma for a dense set $D \subset X$. Note that $X$ must be topologically transitive, so we need only check condition (2) in Lemma \ref{lem_rosendal}. Let $V \ni 1_G$ and let $U \subseteq X$ be open and nonempty. By density, find $x \in D \cap U$. Let $V' \ni 1_G$ be open, symmetric such that $(V')^2 \subset V$. Then let $U'$ be such that $V' \cdot x$ is dense in $U'$. This $U'$ works to witness condition (2). 
\end{proof}

Now we move to the situation of interest. Let $\gamma_n: H_n \times X_n \to X_n$ be a sequence of group actions of Polish groups $H_n$ on Polish spaces $X_n$. For each $n$, let 
\[i_n^{n+1}:H_n \to H_{n+1}\] 
be a continuous group homomorphism and let 
\[f_n^{n+1}: X_n \to X_{n+1}\]
be a continuous injective map which is $i_n^{n+1}$-equivariant, in the sense:
\[f_n^{n+1}(g \cdot x) =i_n^{n+1}(g) \cdot f_n^{n+1}(x)\]
for all $g \in H_n$ and $x \in X_n$.
Then, the direct limit of the actions $\gamma_n$ is the natural action $\gamma$ of the group $\varinjlim (H_n,i_n^{n+1})$ on the space $\varinjlim (X_n,f_n^{n+1})$. We use 
\[i_n:H_n \to \varinjlim (H_n,i_n^{n+1})\] 
and 
\[f_n: X_n \to \varinjlim (X_n,f_n^{n+1})\]
to denote the inclusion maps from each coordinate into the direct limit. The limit object, $\gamma$, is in the category of group actions on sets. However, when $\varinjlim(H_n,i_n^{n+1})$ can be equipped with a group topology $\tau$ that makes every inclusion map $H_n \hookrightarrow \varinjlim (H_n,i_n^{n+1})$ continuous, we get a continuous action of topological group \\
$(\varinjlim(H_n,i_n^{n+1}), \tau)$ on topological space $\varinjlim (X_n,f_n^{n+1})$ where the latter space is equipped with the direct limit topology. It is in general a subtle point when direct limits of topological groups can be equipped with such topologies (see for instance \cite{tatsuma}) but for the setting we consider it will be immediate. Note that neither $\varinjlim (H_n,i_n^{n+1})$ nor $\varinjlim (X_n,f_n^{n+1})$ is necessarily Polish.

Let $G$ be a Polish group and $X$ a Polish $G$-space. We say that a direct limit $\gamma$ of a sequence of actions $\gamma_n:H_n \times X_n \to X_n$ of Polish groups on Polish spaces \emph{approximates the $G$-space $X$} if there is a group homomorphism 
\[i:\varinjlim(H_n,i_n^{n+1}) \to G\]
with dense image and with each $i \circ i_n$ continuous and a continuous map 
\[f:\varinjlim (X_n,f_n^{n+1}) \to X\]
which is equivariant in the sense that $f(h \cdot x)=  i(h) \cdot f(x)$. 

For a metric space $X$, point $x \in X$, and $\epsilon >0$, we use $B_X(x,\epsilon)$ to denote as usual the open ball around $x$ of radius $\epsilon$. Two pieces of notation: given a continuous map between metric spaces $f:(X,d_X) \to (Y,d_Y)$ and $x \in X$, let
\[\textrm{mod} (f,x,\epsilon)=\sup \left\{\delta \ : \ f[B_X(x,\delta)] \subseteq B_Y(f(x),\epsilon)\right\}\]
and let
\[\textrm{comod} (f,x, \epsilon)= \sup \{\delta \ : \ f[B_X(x,\epsilon)] \supseteq B_Y(f(x),\delta) \cap f[X] \}\]
So, for example, when $f$ is an isometry $\textrm{mod}(f,x,\epsilon)=\textrm{comod}(f,x,\epsilon)=\epsilon$. 

\begin{prop}\label{thm_dl_crit}
    Let $G$ be a Polish group and $X$ a Polish $G$-space. Suppose that $X$ is approximated by a direct limit $\gamma$ of Polish group actions $\gamma_n:H_n \times X_n \to X_n$. Fix metrics inducing the Polish topology on each space. Let $y \in X_0$. Assume that for each $n$, $H_n \cdot f_0^n(y)$ is dense in $X_n$. Let $J$ be a fixed natural number. If for any $\epsilon >0$, any $j \geq J$ and any $g \in H_j$ there exists a sequence $\{\alpha_n\}_{n\geq j}$ satisfying:
    \begin{enumerate}
        \item for any $n \geq j$, $B_{H_n} \left(1, \tmod (i \circ i_n, 1, \epsilon)\right) \cdot i_j^n(g) \cdot f_0^n(y)$ is dense in 
        \[B_{X_n} \left(i_j^n(g) \cdot f_0^n(y), \alpha_n\right)\]
        \item $\inf \{\tcomod (f \circ f_n, i_j^n(g) \cdot f_0^n(y), \alpha_n) \ :\ n \geq j \} >0$
    \end{enumerate}
    then $G \cdot (f\circ f_0(y))$ is comeager in $X$.
\end{prop}

\begin{proof}
    Note first that $i[H] \cdot i\circ i_0(y)$ is dense in $X$. This is because $i[H]$ is dense in $G$, the orbit of $f_0(y)$ under the action of $H$ is dense in $\varinjlim (X_n,f_n^{n+1})$ (this is because $f_0^n(y)$ has dense $H_n$-orbit in $X_n$ for all $n)$, and the action of $G$ on $X$ is continuous. We will now check the condition of Lemma \ref{lem_dense_condition} for the dense set $D= i[H] \cdot f\circ f_0(y)$. It is obvious by definition that $D$ is contained in a single $G$-orbit. 

    Fix $V \ni 1_g$ open and fix $x \in D$. Let $g \in H_j$ such that $x= f \circ f_j (g \cdot f_0^j(y))$. Let $\epsilon >0$ such that $B_G(1,\epsilon) \subset V$. Let $\{\alpha_n\}_{n \in\N}$ be as in the Proposition for $\epsilon$ and $g, j$. Let 
    \[\delta = \inf \{\tcomod (f \circ f_n, i_j^n(g) \cdot f_0^n(y), \alpha_n) \ : \ n \geq j\}\]
    and we have by assumption that $\delta >0$. We claim $V \cdot x$ is dense in $B_X(x,\delta)$. Let $z \in B_X(x,\delta)$ and let $\rho >0$. We will find an element of $V \cdot x$ which is $\rho$-close to $z$. By density of $f[\varinjlim X_n]$ in $X$, we may assume that $z \in f[\varinjlim X_n]$. Let $n$ such that $z=f \circ f_n(z')$ for some $z'$ in $X_n$. We may assume that $n \geq j$. By condition (2) and the definition of $\tcomod$ and the definition of $\delta$, we have that $z' \in B_{X_n} (i_j^n(g) \cdot f_0^n(y), \alpha_n)$. By condition (1) we may find 
    \begin{equation}\label{eqnp1}
    h \in B_{H_n} \left(1, \mod (i \circ i_n,1,\epsilon)\right)
    \end{equation}
    such that 
    \begin{equation}\label{eqnp2}
    h \cdot i_j^n(g) \cdot f_0^n(y) \in B_{X_n}\left(z', \tmod (f \circ f_n, z',\rho)\right)
    \end{equation}
    By \eqref{eqnp1}, $i \circ i_n(h) \in V$. By \eqref{eqnp2} and equivariance, 
    \[i \circ i_n(h) \cdot x = f \circ f_n \left(h \cdot i_j^n(g) \cdot f_0^n(y)\right) \in B_{X}(z,\rho). \]
    
\end{proof}

\section{The universal Knaster continuum}\label{secbasic}
Here we collect various basic facts and notation about Knaster continua that we will need later. For us, $\N=\{0,1,2,\ldots \}$. A \emph{Knaster-type} continuum is a continuum which can be written as an inverse limit $\varprojlim (I_n,f_n)$ where each $I_n=[0,1]$ and each $f_n$ is a continuous open map with $f_n(0)=0$. The universal Knaster continuum is the unique-up-to-homeomorphism Knaster-type continuum which continuously and openly, surjects onto all other Knaster-type continua (see \cite{knasterumf} or \cite{eberhartlattice} for more). We will take the following concrete definition of the universal Knaster continuum. Choose $\{p_i\}_{i\in\N \setminus \{0\}}$ a sequence of prime numbers such each prime number appears infinitely many times. Then, set 
\[K=\varprojlim \left(I_n,T_{p_n}\right)\]
with $T_{p_n}:I_n \to I_{n-1}$. Throughout, $T_d$ for $d\geq 1$ is the \emph{standard degree-$d$ tent map} defined by
\[ T_d(x)= \begin{cases}
dx-m & \textrm{ if } x\in \left[\frac{m}{d},\frac{m+1}{d}\right] \textrm{ and $m$ is even}\\

1+m-dx & \textrm{ if } x\in \left[\frac{m}{d},\frac{m+1}{d}\right] \textrm{ and $m$ is odd}\\
\end{cases}\]
By Theorem 2.2 of \cite{eberhartlattice}, $K$ as defined above is the universal Knaster continuum. For each $n \in \N$, we denote by $\pi_n$ the $n$th coordinate projection map $K \to I_n$. We use the following metric on $K$ for convenience in later computations:
\[d_K\left((x_n)_{n\in\N},(y_n)_{n\in\N}\right) = \frac{\norm{x_0-y_0}}{2}+\sum_{i=1}^\infty \frac{1}{\Pi_{j=1}^i p_j} \norm{x_i-y_i}\]

An approximation theorem of D{\k e}bski \cite{debski} provides a useful notion of \emph{degree} for homeomorphisms of $K$. Given $f \in \Homeo (K)$, by D{\k e}bski's theorem there exists a sequence of continuous, open maps $f_n:I_n \to [0,1]$ such that 
\[\sup_{x \in K} \norm{f_n \circ \pi_n(x)-\pi_0 \circ f(x)} \to 0 \textrm{ as } n \to \infty\]
The \emph{degree} of $f$ is defined by the formula:

\begin{equation}\label{eqn_deg}
    \deg(f) := \lim_{n \to \infty} \frac{\deg(f_n)}{\prod_{i=1}^np_i}
\end{equation}
For a continuous open map $g:[0,1] \to [0,1]$, $\deg (g)$ is the minimal $k \in \N$ such that there exists 
\[0=i_0 <i_1<i_2<\cdots <i_k=1\]
such that $g\restriction_{[i_j,i_{}j+1}]$ is monotone for each $j=0,1,\cdots k-1$.
Note that the sequence $\frac{\deg(f_n)}{\prod_{i=1}^np_i}$ as in \eqref{eqn_deg} is eventually constant (see \cite{debski}). A consequence of D{\k e}bski's results, reformulated appropriately, is that the degree map is a continuous group homomorphism from $\Homeo(K)$ onto the group $\Q^\times$ of positive rationals with multiplication (see \cite{knasterumf} for more explanation). 

From now on, $\Homeo^1(K) =\ker (\textrm{deg})$ is the open normal subgroup of $\Homeo(K)$ consisting of all degree-one homeomorphisms of the universal Knaster continuum. We take this group equipped with the metric:
\[d_{\Hk}(f,g)=\sup_{x \in K} d_K(f(x),g(x))\]
We will use the group $\Ho$ of order-preserving homeomorphisms of the unit interval. The metric we will use on $\Ho$ is:

\[\dsup (f,g)=\sup_{x\in [0,1]} \norm{f(x)-g(x)}\]
With the metrics given above, $\Hk$ and $\Ho$ are both non-locally compact Polish groups. Throughout, we say that $g$ is a \emph{generic element} of a Polish group $G$ if the conjugacy class of $g$ is comeager in $G$.

We say that a homeomorphism $f \in \Homeo(K)$ is \emph{diagonal} if for every $n$ there is a continuous open map $f_n:I_{j_n} \to I_n$ where $j_n\geq n$ such that $\pi_n \circ f=f_n \circ \pi_{j_n}$. In the definition, we may assume that $j_{n-1}<j_n$ for each $n$. Diagonal maps are studied by Eberhart, Fugate, and Schumann in \cite{eberhart}; we will use one of their main results in Proposition \ref{prop0} below.

We first concretely investigate the diagonal maps a bit further. If $f$ is diagonal and of degree-one and $f_n:I_{j_n}\to I_n$ as in the definition of diagonal, notice that $\deg (f_n)=\deg (T_{p_{n+1}} \circ T_{p_{n+2}}\circ \cdots \circ T_{p_{j_n}})$. 

For $g \in \Ho$, let $\tilde{g}$ be the homeomorphism defined by $\tilde{g}(x)=1-g(1-x)$. Observe the following:

\begin{lem}\label{lem1}
The map $f \mapsto \tilde{f}$ is an isometry and a continuous group monomorphism 
\[\Ho \to \Ho\]   
\end{lem}

For $f_0,f_1,\ldots, f_{n-1} \in \Ho$, define $\oplus (f_0,f_1,\ldots,f_{n-1}) \in \Ho$ by the following formula:
\[\oplus (f_0,f_1,\ldots, f_{n-1})(x) = \frac{1}{n} f_i(nx-i)+\frac{i}{n} \textrm{ for $x$ in } \left[\frac{i}{n}, \frac{i+1}{n}\right] \]
Given $g \in \Ho$ and $d \in \N$ we define
\[\oplus^d(g)=\oplus (g_0,g_1,\ldots,g_{n-1})\]
where $g_i=g$ when $i$ is even and $g_i=\tilde{g}$ when $i$ is odd.
By direct computation one checks the following: 

\begin{lem}\label{lem0}
    Suppose that $g \in \Homeo_+[0,1]$ and $d \in \N$. We have that
    \[g \circ T_d=T_d \circ \oplus^d(g)\]
\end{lem}

The next lemma allows us to ``straighten'' diagonal maps:

\begin{lem}\label{compltriangle}
    Let $f,g: [0,1] \to [0,1]$ be continuous, open maps, with $f(0)=g(0)=0$. There exists $h \in \Ho$ such that $f =g\circ h$.
\end{lem}

\begin{proof}
    Let $n=\deg(f)=\deg(g)$. Let $I_1,I_2,\ldots,I_n$ be closed intervals with:
    \begin{enumerate}
        \item $\bigcup_{j=1}^nI_j =[0,1]$
        \item $0 \in I_1$
        \item $I_j$ and $I_{j+1}$ intersect in a singleton for each $j$ 
        \item $I_j \cap I_k =\emptyset$ when $\norm{j-k} \geq 2$
        \item $f\restriction_{I_j}$ is monotone for all $j$
    \end{enumerate}
    
    Let $K_1,k_2,\ldots, K_n$ be closed intervals witnessing analogous conditions for $g$. For each $j$, $f \restriction_{I_j}:I_j \to I$ and $g \restriction_{K_j}$ are homeomorphisms and further they are either both order-preserving or both order-reversing. Define $h:[0,1] \to [0,1]$ by the condition that $h \restriction_{I_j} = \left(g\restriction_{K_j}\right)^{-1} \circ f\restriction_{I_j}$ and then one may check that $h \in \Ho$ and that $f=g \circ h$. 
\end{proof}

Notice Lemma \ref{compltriangle} and Lemma \ref{lem0} imply that when $f \in \Homeo^1(K)$ is diagonal, there exists $N \in \N$ so that for all $n \geq N$, there exists $f_n \in \Ho$ such that $\pi_n \circ f= f_n \circ \pi_n$. In this situation, we say that $f$ is \emph{induced at coordinate $n$} by $f_n$. The next proposition summarizes how the diagonal maps sit within $\Hk$; the density part of Proposition \ref{prop0} is due to \cite{eberhart}.

\begin{prop}\label{prop0}
    The diagonal maps form a dense, meager subgroup in $\Homeo^1(K)$.
\end{prop}

\begin{proof}
    It is immediate from the definition that the composition of two diagonal maps is diagonal. Say $f$ is diagonal and let $N$ be so that for all $n\geq N$ there exists $f_n \in \Ho$ with $\pi_n \circ f=f_n \circ \pi_n$. Then it is easy to check by Lemmas \ref{lem1} and \ref{lem0} that $f^{-1}$ is the map satisfying $\pi_n \circ f^{-1}=f_n^{-1}\circ \pi_n$ and $f^{-1}$ is diagonal.

    Density of the subgroup of diagonal maps is a consequence of Theorem 4.7 of \cite{eberhart}. It can also be seen as a consequence of \cite{knasterumf} Theorem 5.1 along with the observation that automorphisms of the pre-universal Knaster continuum induce diagonal homeomorphisms of $K$.

    Let $D_n=\{f \in \Hk \ : \ \exists f_n \in \Ho (\pi_n \circ f=f_n \circ \pi_n\}$ and note that the diagonal maps are exactly $\bigcup_{n\in\N}D_n$. We show that each $D_n$ is nowhere dense. Suppose $U \subset \Hk$ is open and non-empty; we show that $D_n$ is not dense in $U$. Let $g \in U$ be diagonal; say $g$ is induced at coordinate $m$ by $g_m$ where we may assume that $m>n$. Let $\epsilon >0$ such that $B_{\Hk}(g,\epsilon) \subset U$. Let $d=p_{n+1}p_{n+2}\cdots p_m$. Let $h' \in \Ho$ be such that $\dsup(h',g_m)<\frac{\epsilon}{\prod_{i=1}^m p_i}$ and so that $h'\left(\frac{1}{d}\right) \neq \frac{1}{d}$. Let $h \in \Hk$ be induced by $h'$ at coordinate $m$. By Lemma \ref{lem_mod_value}, $h \in U$. Let $\eta <\frac{\norm{h'\left(\frac{1}{d}\right)-\frac{1}{d}}}{2}$ and now we claim that $B_{\Hk}\left(h,\frac{\eta}{\prod_{i=1}^mp_i}\right) \cap D_n =\emptyset$. To see this, suppose $f \in D_n$. Let $f_n \in \Ho$ be such that $f$ is induced by $f_n$ at coordinate $n$. Then $f$ is induced by $\oplus^d(f_n)$ at coordinate $m$ by Lemma \ref{lem0}. Then, notice that since $\oplus^d(f_n) \left(\frac{1}{d}\right)=\frac{1}{d}$, we have that $\dsup \left(\oplus^d(f_n),h'\right) \geq \eta$. This implies that
    \[d_{\Hk} \left(f,h\right) \geq \frac{\eta}{\prod_{i=1}^mp_i} \]
    simply by considering some $x \in K$ where $\pi_k(x)$ is such that 
    \[\norm{\oplus^d(f_n)(\pi_k(x))-h'(\pi_k(x))} \geq \eta\]
    So, $D_n$ is not dense in $U$. 
\end{proof}

Notice the Baire category theorem implies that there are homeomorphisms of $K$ which are not diagonal; in fact, a similar argument shows there are homeomorphisms of $K$ of any degree which are not diagonal--this provides another proof of and a slight strengthening of a result of Eberhart, Fugate, and Schumann (\cite{eberhart}, p.138).

\section{The conjugacy action of $\Ho$}\label{sec:homeo_int}
In this section we establish the crucial lemma about the conjugacy action of $\Ho$. 

\subsection{The generic element of $\Ho$}
Let $f \in \Homeo_+[0,1]$. By $\fix (f)$, we denote the set of fixed points of $f$; it is a closed subset of $[0,1]$ which includes the points 0 and 1. The complement of $\fix(f)$ is a countable union of disjoint open intervals. This countable collection of intervals have an obvious order on them coming from the order on $[0,1]$. For each interval $I$ in this collection, continuity of $f$ implies that either we have that for all $x \in I$, $f(x)>x$ or we have that for all $x \in I$, $f(x)<x$. In the former case we call $I$ a \emph{positive interval} and in the latter, a \emph{negative interval}. So, we can associate to $f$ a countable linear order $L_f$ along with a marking map $m_f:L_f \to \{-1,1\}$ that encodes the order-type and positive versus negative orientation of the intervals forming $\fix(f)^c$. The following theorem appears in \cite{hjorth}:

\begin{thm}[\cite{hjorth}, Theorem 4.6]
    Two homeomorphisms $f$ and $g$ are conjugate in $\Homeo_+[0,1]$ if and only if there is an order-preserving bijection $\phi:L_f \to L_g$ such that $m_g(\phi(x))=m_f(x)$ for all $x \in L_f$. 
\end{thm}

So in particular one may concretely describe the generic element of $\Ho$:

\begin{thm}[\cite{kr}, Theorem 5.3;  \cite{glasnerweiss}, Section 9]\label{thm2}
Any $f \in \Homeo_+[0,1]$ such that:
\begin{enumerate}
    \item $L_f$ is a dense linear order without endpoints (i.e., order-isomorphic to $\Q$)
    \item $\fix (f)$ is totally disconnected
    \item $m_f^{-1}(1)$ and $m_f^{-1}(-1)$ are each dense, unbounded in both directions subsets of $L_f$ 
\end{enumerate}
has comeager conjugacy class in $\Homeo_+[0,1]$. 
\end{thm} 

Our $\oplus$ operation preserves the conditions of Theorem \ref{thm2}:

\begin{lem}\label{lem2}
    If $f \in \Ho$ satisfies (1)-(3) in Theorem \ref{thm2}, then so does $\oplus^d(f)$ for any $d \in \N$.  
\end{lem}

\begin{proof}
Let $f$ satisfy (1)-(3). It is readily checked that $\fix(\tilde{f}) =\{1-x \ : \ x \in \fix (f) \}$. So, $\fix (\tilde{f})$ is totally disconnected and the intervals forming $L_{\tilde{f}}$ form a dense linear order without endpoints. We have also that $f(x)<x$ if and only if $\tilde{f}(1-x)>1-x$ and $f(x)>x$ if and only if $\tilde{f}(1-x)<1-x$. This implies that $\tilde{f}$ satisfies condition (3) above.

The structure $L_{\oplus^d(f)}$ is exactly a copy of $L_f$ followed by a copy of $L_{\tilde{f}}$ followed by a copy of $L_f$ and so on, with $d$ copies in total. So in particular it satisfies (1) and (3) above. To see (2) note that $\fix\left(\oplus^d(f)\right)$ is exactly
\[\bigcup_{i=0}^{d-1} \left\{\frac{x+i}{n} \ : \ x \in \fix (f_i)\right\}\]
where $f_i=f$ for $i$ even and $f_i=\tilde{f}$ for $i$ odd. 
\end{proof}

Notice that Lemma \ref{lem2} immediately implies that $\Homeo^1(K)$ has a dense conjugacy class.

\begin{prop}\label{propdense}
Let $f \in \Hk$ and $f_0 \in \Ho$ such that $\pi_0 \circ f= f_0 \circ \pi_0$ and with $f_0$ satisfying (1)-(3) above. Then, the conjugacy class of $f$ is dense in $\Homeo^1(K)$. Actually, if $\mathcal{C}$ is the conjugacy class of $f$, then the intersection of $\mathcal{C}$ with the set of diagonal maps is dense in $\Homeo^1(K)$.
\end{prop}

\begin{proof}
Fix $y \in \Homeo^1(K)$ and fix $\eta >0$. By Proposition \ref{prop0}, we may assume that $y \in \mathcal{I}$. Say, $y$ is induced by $y_m \in \Ho$ at coordinate $m$. We may also assume that $m$ is large enough so that
\[\sum_{i=m+1}^\infty \frac{1}{\prod_{j=1}^ip_j} <\frac{\eta}{2}\]
By Lemma \ref{lem2} and Theorem \ref{thm2}, we have that $f_m$, the map that induces $f$ at coordinate $m$ has dense conjugacy class in $\Ho$. Let $\epsilon>0$ be such that
\[\left(\sum_{i=1}^m \frac{\prod_{j=i}^m p_j}{\prod_{j=1}^i p_j}\right)\epsilon<\frac{\eta}{2}\]
and find $g_m \in \Ho$ such that $\dsup(g_m^{-1}f_mg_m, y_m)<\epsilon$. Then, if $g$ is induced at coordinate $m$ by $g_m$, it follows that $d_G(g^{-1}fg,y)<\eta$. 
\end{proof}

\subsection{Controlling generic elements}
A consequence of Effros' Theorem \cite{effros} is that if $g$ is a generic element in a Polish group $G$, then for any neighborhood $V$ of $1_G$, we have that $\{h^{-1}gh \ : \ h \in V\}$ is dense in (actually, comeager in) some neighborhood of $g$. Apriori we have no control of \emph{how large} a neighboorhood of $g$ the set above is dense in. The main technical lemma below shows that we can exert some control with respect to operation $\oplus$:

\begin{lem}\label{mainlem}
    Let $B_{\Ho}(1,\epsilon)$ for some $\epsilon >0$ be an open ball centered at the identity in $\Homeo_+[0,1]$. Let $f \in \Homeo_+[0,1]$ be such that $f$ has comeager conjugacy class in $\Homeo_+[0,1]$. Suppose that $\delta >0$ is so that 
    \begin{enumerate}
        \item  $\{g^{-1}fg \ : \ g \in B_{\Ho}(1,\epsilon)\}$ is dense in $B_{\Ho}(f,\delta)$
        \item any $f' \in B_{\Ho}(f,\delta)$ has a fixed point less than $\epsilon$ 
        \item any $f' \in B_{\Ho}(f, \delta)$ has a  fixed point greater than $1-\epsilon$
    \end{enumerate}
    
  \noindent  Then for any $d \in \N$: 
    \[\left\{g^{-1}(\oplus^d(f))g \ : \ g \in B_{\Ho}\left(1,\frac{\epsilon}{d}\right)\right\}\]
    is dense in $B_{\Ho}\left(\oplus^d(f), \frac{\delta}{d}\right)$.
\end{lem}

\begin{proof}
    We do the proof in detail for $d=2$ and then indicate how the argument generalizes. Fix $f, \epsilon, \delta$ satisfying the assumptions of the lemma. We want to show that 
    \[\left\{g^{-1}(f\oplus \tilde{f})g \ : \ g \in B\left(1,\frac{\epsilon}{2}\right)\right\}\]
    is dense in $B\left(f\oplus \tilde{f}, \frac{\delta}{2}\right)$. So let $h \in B\left(f\oplus \tilde{f},\frac{\delta}{2}\right)$ and let $\eta >0$. We find $g \in B \left(1,\frac{\epsilon}{2}\right)$ so that 
    \[d(g^{-1} (f \oplus \tilde{f}) g, h) <\eta\]

    \textbf{Case 1:}
    First, suppose that $h\left(\frac{1}{2}\right)=\frac{1}{2}$. Then, consider $h_1:[0,1] \to [0,1]$ given by:
    \[h_1(x):= 2 h\left(\frac{x}{2}\right)\]
    and $h_2:[0,1] \to [0,1]$ given by:
    \[h_2(x)=2h\left(\frac{x+1}{2}\right)-1\]
    Since $d(f \oplus \tilde{f}, h) <\frac{\delta}{2}$, one can check that $d(f,h_1)<\delta$ and $d(\tilde{f},h_2)<\delta$. So, by our assumption on $f$ we may find $g_1 \in B\left (1,\epsilon\right)$ such that 
    \[d(g_1^{-1}fg_1,h_1)<\eta.\]
    As $d(f, \tilde{h_2})<\delta$ (by Lemma \ref{lem1}), we find $g_2 \in B(1,\epsilon)$ such that $d(g_2^{-1}fg_2,\tilde{h_2})<\eta$ and thus
    \[d(\tilde{g_2}^{-1}\tilde{f}\tilde{g_2},h_2)<\eta\]
    Now, one checks that $g_1 \oplus \tilde{g_2} \in B\left(1,\frac{\epsilon}{2}\right)$ and that setting $g=g_1 \oplus \tilde{g_2}$ we have that
    \[d(g^{-1} (f \oplus \tilde{f}) g, h)<\eta\]

\textbf{Case 2:}
    Now we assume that $h \left(\frac{1}{2}\right) \neq \frac{1}{2}$. Let $a<\frac{1}{2}<b$ such that $h(a)=a$, $h(b)=b$ and for all $x \in (a,b)$, $h(x) \neq x$. We will assume that $h(x)>x$ for $x \in (a,b)$ (the case where $h(x)<x$ for $x \in (a,b)$ is analogous). 

    By conditions (2) and (3) on $\delta$, notice that we have that $\frac{1}{2}-a <\frac{\epsilon}{2}$ and $b-\frac{1}{2}<\frac{\epsilon}{2}$. Let $h' \in \Homeo_+[0,1]$ be such that $h'=h$ on $[0,a] \cup [b,1]$, $h'\left(\frac{1}{2}\right)=\frac{1}{2}$, $x<h'(x)\leq h(x)$ on $(a,b)\setminus \left\{\frac{1}{2}\right\}$ and $h' \in B\left(f\oplus \tilde{f},\frac{\delta}{2}\right)$. Applying Case 1, we may find $g=g_1 \oplus g_2 \in B\left(1,\frac{\epsilon}{2}\right)$ such that $d(g^{-1}(f \oplus \tilde{f})g,h') <\eta$ and by approximating $h'$ sufficiently well we may assume $d(g^{-1}(f \oplus \tilde{f})g,f \oplus \tilde{f})<\frac{\delta}{2}$. Note that $h'$ has a fixed point at $a$ and a fixed point at $b$; by approximating $h'$ sufficiently well using case 1, we may also assume that $g^{-1}(f \oplus \tilde{f})g$ has fixed points $a'$ and $b'$ with:
    \begin{enumerate}
        \item $\norm{a'-\frac{1}{2}}<\frac{\epsilon}{2}$
        \item $\norm{b'-\frac{1}{2}}<\frac{\epsilon}{2}$
        \item $\norm{a'-a}<\eta$ and $\norm{b'-b}<\eta$
        \item $\norm{h'(x)-x}<\eta$ and $\norm{h(x)-x}<\eta$ on the interval between $a'$ and $a$
        \item $\norm{h'(x)-x}<\eta$ and $\norm{h(x)-x}<\eta$ on the interval between $b'$ and $b$  
         \end{enumerate}
    
 Set $h''=g^{-1}(f \oplus \tilde{f})g$. Consider $g(a')$ and $g(b')$. These are fixed points of $f \oplus \tilde{f}$. Further, $g(a')<\frac{1}{2}$ and $g(b')>\frac{1}{2}$ because $g$ is a direct sum of two maps. 

    We will define a map $g':[0,1] \to [0,1]$. First, let $g'=g$ on $[0,a'] \cup [b',1]$. We need to define $g'$ as a map $[a',b'] \to [g(a'),g(b')]$. Consider $f \oplus \tilde{f} \restriction [g(a'),g(b')]$. Find a positive interval $(i,j)$ of $f \oplus \tilde{f}$ close enough to $\frac{1}{2}$ so that $j-a'<\frac{\epsilon}{2}$ and $b'-i<\frac{\epsilon}{2}$ (this is possible by condition (3) in Theorem \ref{thm2} applied to both $f$ and $\tilde{f}$). Let $k$ be a homeomorphism from $[a',b'] \to [a',b']$ such that
    \begin{enumerate}
        \item if $a<a'$ and $b'<b$, then $k = h$ on $[a'',b'']$ for some $a''>a'$ and $b''<b'$ such that 
        \begin{enumerate}[(a)]
            \item $\norm{h(x)-x}<\eta$ on $(a,a'')$ and $(b'',b)$
            \item $\norm{a''-a}, \norm{b''-b}<\eta$
            \item $\norm{g(a')-a''}<\frac{\epsilon}{2}$ and $\norm{g(b')-b''}<\frac{\epsilon}{2}$
        \end{enumerate}
         
        \item if $a'<a$ and $b<b'$, then $k=h$ on $[a,b]$
        \item if $a<a'$ and $b<b'$, then $k=h$ on $[a'',b]$ where $a''$ is chosen as in case (1)
        \item if $a'<a$ and $b'<b$, then $k=h$ on $[a,b'']$ where $b''$ is chosen as case (1)
    \end{enumerate}
    \emph{and}, $L_k$ is isomorphic to $L_{f \oplus \tilde{f} \restriction[g(a'),g(b')]}$ via an isomorphism that sends interval $J$ to $(i,j)$ where $J \in L_k$ is the interval containing $(a'',b'')$ in case (1) above, $J$ is the interval containing $(a,b)$ in case (2) above, $J$ is the interval containing $(a'',b)$ in case (3) and $J$ is the interval containing $(a,b'')$ in case (4). Now, we let $g'$ be such that $(g')^{-1}(f \oplus \tilde{f})\restriction_{[g(a'),g(b')]} g'=k\restriction_{[a',b']}$ and $g'$ sends $J$ to $(i,j)$. One checks that $d((g')^{-1}(f\oplus \tilde{f})g',h)<\eta$ and by our choice of interval $(i,j)$, for any $x \in [a',b']$, $\norm{g'(x)-x}<\frac{\epsilon}{2}$. 

    When $d>2$, the proof follows analogously. We first address the case that $h$ is such that $h \left(\frac{i}{d}\right)=\frac{i}{d}$ for all $0 \leq i\leq d$. Then, in the general case, just as in case 2 above we first replace $h$ by an $h'$ so that $h' \left(\frac{i}{d}\right)=\frac{i}{d}$ for each $0\leq i\leq d$ and $h'$ is very close to $h$ except possibly in small neighborhoods (radius less than $\frac{\epsilon}{d}$) around $\frac{i}{d}$. Then we let $g$ be such $g\oplus^d(f)g^{-1}$ is within $\eta$ of $h'$. We modify $g$ on the small intervals about each $\frac{i}{d}$ independently, just as done in the $d=2$ case, to produce $g'$ and then we check that indeed $g'$ conjugates $\oplus^d(f)$ to be close to $h$ and that $g' \in B_{\Ho}\left(1,\frac{\epsilon}{d}\right)$.
\end{proof}

\section{The main theorem}\label{sec:main_thm}
\subsection{Some set-up}
Recall that
\[K=\varinjlim (I_n,T_{p_n})\]
is the universal Knaster continuum, with each $I_n=[0,1]$, each $T_{p_n}:I_n \to I_{n-1}$ the standard degree-$p_n$ tent map, and $\{p_n\}_{n\in\N \setminus \{0\}}$ a sequence of prime numbers. Throughout this section we use the metric $d_{\Hk}$ on $\Hk$ defined in Section \ref{secbasic} and the usual supremum metric $d_{\sup}$ on $\Ho$. 

Let $H$ be the group $H=\varinjlim (H_n,i_n^{n+1})$ where $H_n=\Ho$ for each $n \in \N$ and where $i_n^{n+1}:H_n \to H_{n+1}$ is given by $i_n^{n+1}(f)=\oplus^{p_{n+1}}(f)$. We define $i:H \to \Hk$ as follows. For $g \in H$, say $g=i_j(g_j)$ for some $g_j \in H_j$, we define $i(g)$ to be the diagonal map satisfying 
\[\pi_j \circ i(g)= g_j \circ \pi_j\]
Lemma \ref{lem0} implies that $i$ is well-defined. 

Note that, by Proposition \ref{prop0}, the direct limit of the actions of $H_n$ on itself by conjugacy approximates the action of $\Hk$ on itself by conjugacy. 

\subsection{Two lemmas}
The first lemma computes $\tmod$-values for $i \circ i_n$.

\begin{lem}\label{lem_mod_value}
    Let $H_n$, $i_n$ and $i$ be as above. For any $g \in H_n$ and any $\epsilon >0$:
    \[\tmod ( i\circ i_n, g, \epsilon) \geq \frac{\epsilon}{\prod_{i=1}^np_i}\]
\end{lem}

\begin{proof}
    Fix $\epsilon >0$. Let $h \in B_{\Ho} \left(g, \frac{\epsilon}{\prod_{i=1}^n p_i} \right)$. We need to show: 
    
    \[i \circ i_n(h) \in B_{\Hk} (i \circ i_n(g), \epsilon).\]
    
    Observe first that for any $g_1,g_2 \in \Ho$:
    \[d_{\sup}(g_1,g_2) <\epsilon \implies d_{\sup} \left(\oplus^d(g_1),\oplus^d(g_2)\right) <\frac{\epsilon}{d}\]

    By Lemma \ref{lem0}, for $N >n$
    \begin{equation} \label{eqn_mod_value_1}
    d_{sup} \left(i_n^N (g), i_n^N(h) \right) <\frac{\epsilon}{\prod_{i=1}^N p_i}
    \end{equation}

    Let $x = (x_n)_{n \in \N} \in K$. To ease notation, let $y = i \circ i_n(g)(x)$ and $z= i \circ i_n(h)(x)$. By \eqref{eqn_mod_value_1}, we have that for $m \geq n$:
    \[\norm{y_m-z_m} < \frac{\epsilon}{\prod_{i=1}^m p_i}\]
    For $m <n$ we have:
    \begin{align*}
        \norm{ y_m - z_m} &= \norm{T_{p_{m+1}p_{m+2}\cdots p_n}(y_n)- T_{p_{m+1}p_{m+2}\cdots p_n} (z_n)}\\
        & < p_{m+1}p_{m+2} \cdots p_n \frac{\epsilon}{p_1\cdots p_n}\\
        &= \frac{\epsilon}{p_1p_2\cdots p_m}
    \end{align*}

    So we have:
    \[d_K(y,z) \leq \frac{\epsilon}{2} +\sum_{m=1}^\infty \frac{1}{p_1p\cdots p_m}\frac{\epsilon}{p_1\cdots p_m} \leq \epsilon \left(\frac{1}{2} +\sum_{m=1}^\infty \frac{1}{2^{2m}}\right) <\epsilon\]

    and therefore
    \[d_{\Hk} (i \circ i_n(g),i \circ i_n(h)) <\epsilon \] 
\end{proof}

The next lemma is a key fact about about tent maps needed for one of the bounds in the main result.

\begin{lem}\label{lem3}
Let $\delta <\frac{1}{4}$. If $f,g \in \Homeo_+[0,1]$ are such that $\dsup(f,g) \geq \frac{\delta}{d}$, then there exists $x \in [0,1]$ such that at least one of the following two conditions hold:
\begin{enumerate}
    \item $\norm{T_d \circ f(x)-T_d \circ g(x)} \geq \delta$
    \item $\norm{T_d \circ f(x)-T_d \circ g(x)} \geq \frac{\delta}{2}$ and at least one of $T_d \circ f(x)$ or $T_d \circ g(x) \in \{0,1\}$
\end{enumerate}

\end{lem}

\begin{proof}
Consider $[0,1]$ as divided into sub-intervals $I_j :=\left[\frac{j}{d},\frac{j+1}{d}\right]$ for $j=0,1,\ldots, d-1$. 

    Since $\dsup(f,g) \geq \frac{\delta}{d}$ and $[0,1]$ is compact, there exists $x \in [0,1]$ such that $\norm{f(x)-g(x)} \geq \frac{\delta}{d}$. Notice that if $f(x)$ and $g(x)$ are in the same sub-interval $I_j$, then $\norm{T_d \circ f(x)-T_d \circ g(x)}=d\norm{f(x)-g(x)} \geq \delta$ and we are done as we have satisfied point (1). 

    So we may assume that $f(x)$ and $g(x)$ are in different sub-intervals.

    Without loss of generality suppose $f(x) <g(x)$. As $f(x)$ and $g(x)$ are in different sub-intervals, there exists some $j$ such that:
    \[\frac{j}{d}-f(x) \geq \frac{\delta}{2d} \textrm{ or } g(x)-\frac{j}{d} \geq \frac{\delta}{2d}\]
    
    \textbf{Case 1:} Suppose that $g(x)-\frac{j}{d} \geq \frac{\delta}{2d}$. Then, let $x' \in [0,1]$ be such that $f(x') =\frac{j}{d}$. Note that $g(x')>g(x)$ and so $g(x')$ is at distance greater than $\frac{\delta}{2d}$ from $f(x')$. If $g(x') \in I_j$, then we are done. In fact, notice that $\norm{T_d \circ f(x')-T_d \circ g(x')} <\frac{\delta}{2}$ if and only if 
     \[g(x') \in \left(\frac{k_1}{d}-\frac{\delta}{2d},\frac{k_1}{d}-\frac{\delta}{2d}\right)\]
     for some $k_1>j$ such that $k_1 = j \pmod 2$. If this happens then set $k_1$ as above. Now, let $x''$ be such that $f(x'')=\frac{k_1-1}{d}$. As before either it is the case that $\norm{T_d \circ f(x'')-T_d \circ g(x'')} \geq \frac{\delta}{2}$ or there exists $k_2$ such that $k_2 =k_1-1\pmod 2$, $k_2 \geq k_1$ and
     \[g(x'') \in \left(\frac{k_2}{d}-\frac{\delta}{2d},\frac{k_2}{d}-\frac{\delta}{2d}\right)\]
     Notice that the fact that $k_2 \geq k_1$ and $k_2 =k_1-1 \pmod 2$ actually implies that $k_2 >k_1$. Continuing in this way, we define a sequence $k_1<k_2<\cdots$ and either at some stage we find an $x$ witnessing point (2) of the lemma or we continue until we have reached some $k_n=d$. That is, there is some $y$ such that $f(y)=\frac{k_{n-1}-1}{d}$ and 
     \[g(y) \in \left(1-\frac{\delta}{2d}, 1\right)\]
     and $k_n =k_{n-1}-1 \pmod 2$. In particular, $k_{n-1}-1 < k_n-1$. Now, let $y'$ be such that $f(y')=\frac{k_n-1}{d}$ and notice that $g(y')>g(y)$ and so
     \[g(y')\in \left(1-\frac{\delta}{2d}, 1\right) \]
     If $T_d\left(\frac{d-1}{d}\right)=1$, then we have:
     \[T_d\circ f(y')=T_d \left(\frac{d-1}{d}\right)=1\]
     and since $T_d (1)=0$: 
     \[T_d\circ g(y') <\frac{\delta}{2}\]
     Since $1-\frac{\delta}{2}>1-\frac{1}{8}>\delta$, (we are using the assumption here that $\delta <\frac{1}{4}$), $y'$ witnesses the lemma. A totally analogous argument works if $T_d \left(\frac{d-1}{d}\right)=0$.

     \textbf{Case 2:} Suppose that $\frac{j}{d}-f(x) \geq \frac{\delta}{2d}$. This case proceeds in an analogous way as Case 1 with the roles of $f$ and $g$ switched; the sequence $k_1,k_2,\ldots $ is strictly decreasing instead of increasing but the arguments are exactly the same. 

     Notice that the $x$ chosen by this procedure to witness the fact that 
     \[\norm{T_d\circ f(x)-T_d\circ g(x)}\geq \frac{\delta}{2}\]
     has the property required by point (2).
\end{proof}

\subsection{Proof of main theorem}
Now we prove the main theorem:

\begin{thm}\label{thm_main}
    The group $\Hk$ has a comeager conjugacy class.
\end{thm}

\begin{proof}
    We will check the conditions of Proposition \ref{thm_dl_crit}. We set $J$ as in the statement of the Proposition to be 2. The $H_n$'s and $i_n^{n+1}, i_n, i$ from the statement of the proposition are as defined above. For each $n$, $X_n=H_n$ and $f_n^{n+1}=i_n^{n+1}$. The action of $H_n$ on itself is the conjugacy action. Let $y= i \circ i_0 (\phi)$ where $\phi \in \Ho$ has comeager conjugacy class in $\Ho$. By Lemma \ref{lem2}, $f_0^n(\phi)$ has comgeager conjugacy class in $H_n$ for all $n \in \N$. Fix $\epsilon >0$ and $g \in H_j$ for some $j$. Since $g \cdot \phi$ has comeager conjugacy class in $\Ho$, we know that 

    \[B_{\Ho} \left(1, \frac{\epsilon}{\prod_{i=1}^jp_i}\right) \cdot g \cdot \phi \textrm{ is dense in } B_{\Ho}(g \cdot \phi, \delta)\]
for some $\delta >0$. By making $\delta$ smaller, we may assume that any $h \in B_{\Ho} \left(g \cdot \phi, \delta\right)$ has a fixed point less than $\epsilon$ and a fixed point greater than $1-\epsilon$. We also assume (by making $\delta$ smaller if necessary) that $\delta < \min \{\frac{1}{4}, \frac{1}{p_j}\}$.
    
    Let $\alpha_j = \delta$ and $\alpha_n = \frac{\delta}{p_{j+1} \cdots p_n}$ for each $n >j$. 

    By choice of $\delta$, Lemma \ref{mainlem} immediately implies condition (1) of Proposition \ref{thm_dl_crit}. So it remains to check (2). In particular, we show
    \begin{equation} \label{eq_main_thm_1}
    \tcomod (i \circ i_n, i_j^n(g) \cdot i_0^n(\phi), \alpha_n) >\frac{\delta}{p_1p_2 \cdots p_j}
    \end{equation}
    for each $n \geq j$, which implies (2). 

    First for $n=j$. We need to show:
    \[B_G \left( i \circ i_j(g \cdot \phi), \frac{\delta}{p_1\cdots p_j} \right) \cap i \circ i_j[H_j] \subseteq i \circ i_j \left[B_{\Ho}\left(g \cdot \phi, \delta\right)\right]\]

    Let $p= i \circ i_j (p')$ with $d_{\sup}(p', g \cdot \phi) \geq \delta$. We claim $d_G(p, i \circ i_j(g \cdot \phi)) \geq  \frac{\delta}{p_1 \cdots p_j}$; this follows from considering any point $x= (x_n)_{n \in \N} \in K$ such that $\norm{p'(x_j)-g \cdot \phi(x_j)} \geq \delta$ and noting that
    \[d_K\left(p(x), i \circ i_j(g \cdot \phi)(x) \right) \geq \frac{\delta}{p_1\cdots p_j}\]
    by definition of the metric $d_K$, just by considering coordinate $j$ of $K$. This checks Equation \eqref{eq_main_thm_1} for $n=j$.

    Now for $n >j$ we proceed like this. Let $p = i \circ i_n (p')$ with $p' \in H_n$ and $d_{\sup}(p', i_j^n(g \cdot \phi)) \geq \frac{\delta}{p_{j+1}\cdots p_n}$. To prove that \eqref{eq_main_thm_1} holds we must prove that:
    
    \begin{equation}\label{eqn_main_thm_2}
        d_G\left(p, i \circ i_j(g \cdot \phi)\right) \geq \frac{\delta}{p_1\cdots p_j} 
    \end{equation}

    Apply Lemma \ref{lem3} to $p'$ and $i_j^n(g \cdot \phi)$ with $d= p_{j+1} \cdots p_n$. We have two cases depending on which of (1) or (2) of Lemma \ref{lem3} holds. 

    \textbf{Case 1}: there is $x' \in [0,1]$ such that
    \begin{equation}\label{eqn_main_thm_3}
    \norm{T_d \circ p'(x')- T_d \circ i_j^n(g \cdot \phi)(x')} \geq \delta
    \end{equation}
    Now, take $x= (x_n)_{n \in\N} \in K$ with $x_n=x'$. 

    Recall we use $\pi_k:K \to [0,1]$ as the projection onto the $k$th coordinate for any $k$. Note that
    \[\pi_j (p (x))=T_d \circ p'(x')\]
    and 
    \[\pi_j (i \circ i_j (g \cdot \phi)(x))= T_d \circ i_j^n(g \cdot \phi)(x')\]
    and thus, by considering just coordinate $j$ we have by \eqref{eqn_main_thm_3} that:
    \[d_K\left(p(x), i \circ i_j(g \cdot \phi)(x)\right) \geq \frac{\delta}{p_1\cdots p_j}\]
    which implies \eqref{eqn_main_thm_2} as desired.

    \textbf{Case 2:} there is some $x' \in [0,1]$ such that 
    \begin{equation}\label{eqn_main_thm_4}
        \norm{T_d \circ p'(x') -T_d \circ i_j^n(g \cdot \phi)(x')} \geq  \frac{\delta}{2}
    \end{equation}

    and 

    \begin{equation}\label{eqn_main_thm_4_b}
    \textrm{one of } T_d \circ p'(x') \textrm{ or } T_d \circ i_j^n(g \cdot \phi)(x') \in \{0,1\}
    \end{equation}

    Notice that if $\norm{T_d \circ p'(x') -T_d \circ i_j^n(g \cdot \phi)(x')} \geq \delta$, then we are done by Case 1, so we may assume:

    \begin{equation}\label{eqn_main_thm_5}
        \norm{T_d \circ p'(x') -T_d \circ i_j^n(g \cdot \phi)(x')} < \delta <\frac{1}{p_j}
    \end{equation}

    By \eqref{eqn_main_thm_4_b} and \eqref{eqn_main_thm_5}, we have that $T_d \circ p'(x')$ and $T_d \circ i_j^n(g \cdot \phi)(x')$ are in the same subinterval of the form $\left[\frac{i}{p_j}, \frac{i+1}{p_j}\right]$ of $[0,1]$. So:
    \[\norm{T_{p_j} \circ T_d \circ p'(x')- T_{p_j} \circ T_d \circ i_j^n(g \cdot \phi)(x')} = p_j \norm{T_d \circ p'(x')- T_d \circ i_j^n(g \cdot \phi)(x')} \geq \delta\]
    where the inequality above is by \eqref{eqn_main_thm_4} and the fact that $p_j \geq 2$. 

    Now let $x = (x_n)_{n\in\N} \in K$ with $x_n=x'$. By considering just coordinate $j-1$ (notice here we implicitly use that $j \geq 2$) one gets:
    \[d_K \left(p(x), i \circ i_j (g \cdot \phi)(x) \right) \geq \frac{\delta}{p_1\cdots p_{j-1}} \geq \frac{\delta}{p_1 \cdots p_{j-1}p_j}\]
    and so \eqref{eqn_main_thm_2} holds.   
\end{proof}

\printbibliography

\end{document}